\documentclass{article} 

\usepackage{amsfonts,amssymb,amsmath,amsthm}
\usepackage[all, poly]{xy} 

\title{Centralizer algebras of the group associated to ${\Z}_4$-codes}
\author{Masashi Kosuda and Manabu Oura} 

\theoremstyle{plain}
\newtheorem{thm}{Theorem}[section]
\newtheorem{cor}[thm]{Corollary}

\newtheorem{lem}[thm]{Lemma}
\theoremstyle{definition}

\newtheorem{rem}[thm]{\sc Remark}

\newcommand{\Z}{{\mathbb Z}}
\newcommand{\Comp}{\mathbb C}
\newcommand{\bolde}{\mbox{\boldmath $e$}}

\newcommand{\boldv}{\mbox{\boldmath $v$}}
\newcommand{\boldw}{\mbox{\boldmath $w$}}
\newcommand{\boldX}{\mbox{\boldmath $X$}}

\newcommand{\End}{\mbox{End}}

\newcommand{\bD}{\mbox{$\bar D$}}
\newcommand{\bT}{\mbox{$\bar T$}}

\allowdisplaybreaks[4]

\date{}

\begin{document}
\maketitle
\begin{abstract}
The purpose of this paper is to investigate the finite group which appears in the study of the Type II ${\Z}_4$-codes.
To be precise, it is characterized in terms of generators and relations, and  we determine the structure of the centralizer algebras
of the tensor representations of this group.
\end{abstract}

\renewcommand{\thefootnote}{\fnsymbol{footnote}}
\footnote[0]{
\hspace{0.4cm}
\newline
\hspace{0.4cm}
running head: Centralizer algebra of the group associated to ${\Z}_4$-codes
}

\section{Introduction}

This is a continuation of the paper \cite{KO}
in which the centralizer algebras of the tensor representations
of the specified finite group are investigated.
The group discussed there is closely connected to binary codes,
while the group $\mathfrak{G}$ in the present paper arises from $\Z_4$-codes.
The purpose of this paper is 
to characterize $\mathfrak{G}$ in terms of generators and relations, and 
to determine the structure of the centralizer algebras
of the tensor representations of this group.

One of our interests of this researches
originally lies in the applications of various coding theory to modular forms.
And then we become also interested in the groups appearing in the course of this study.
We shall describe this in $\Z_4$-codes.

The theory of codes over ${\Z}_4$
has attracted great interest since around 1990
({\it cf}. \cite{HZ4}, \cite{BZ4}, \cite{BTypeII}
and the references cited there).
Let ${\Z}_4=\{0,1,2,3\}$ be the residue ring of the rational integers
modulo $4$.
We take a ${\Z}_4$-code $C$ of length $n$, that is,
an additive subgroup of ${\Z}_4^n$.
It is said to be Type II if 
\[
	C=\{u\in {\Z}_4^n\ :\ (u,v)=\sum_i u_iv_i\equiv 0\pmod{4},\ \forall v\in C\}
\]
and
\[
	(v,v)\equiv 0\pmod{8}, \ \ \forall v\in C.
\]
The symmetrized weight enumerator of $C$ is, by definition,
\[
S_C(x,y,z)=\sum_{v\in C}
	x^{wt_0(v)}y^{wt_1(v)+wt_3(v)}z^{wt_2(v)}
\]
where $wt_a(v)=\sharp \{i:\ v_i=a\}$ for $v=(v_1,v_2,\ldots,v_n)$.
Among the several applications of the weight enumerator,
we have been interested in constructing a modular form
(\textit{cf}. \cite{BTypeII,Ru2,Ru3}. See \cite{Cameron} for combinatorial application).
Let 
\[
f_a(\tau)=\sum_{n\in \Z}\exp 2 \pi \sqrt{-1} \cdot 2\tau \left(x+\frac{a}{4}\right)^2
\]
be the modified theta constants where $a\in \{0,1,2\}$ and $\tau$ is a complex number with positive imaginary part.
It is then known that for a Type II $\Z_4$-code $C$ of length $n$,
the weight enumerator $S_C(x,y,z)$ gives a modular form of $n/2$ for $SL(2,\Z)$
by replacing $(x,y,z)$ with $(f_0(\tau),f_1(\tau),f_2(\tau))$.
This remarkable fact comes from 
the invariance property of the symmetrized weight enumerator under the action of the group $\mathfrak{G}$.
The group $\mathfrak{G}$ is defined as follows.

Let $\eta=\dfrac{1+i}{\sqrt{2}}$ be a primitive 8-th root of unity
and $\mathfrak{G}$ a group generated by
${\cal D} = \mbox{diag}(1, \eta, -1)$
and
\[
	{\cal T} =
	\frac{\eta}{2}
	\begin{pmatrix}
		1& 2&1\\
		1& 0&-1\\
		1&-2&1
	\end{pmatrix}.
\]
The group $\mathfrak{G}$ is of order $384$ and naturally acts on the polynomial ring
$\Comp[x,y,z]$ of three variables over the complex number field $\Comp$.
Then we can consider the invariant ring of $\mathfrak{G}$:
\[
	\Comp[x,y,z]^{\mathfrak{G}}
	= \{f\in \Comp[x,y,z]:\
		Af=f\ \mbox{for any}\ A\in \mathfrak{G}\}.
\]
The invariance property mentioned above means that 
the symmetrized weight enumerator for a Type II $\Z_4$-code is an element of $\Comp[x,y,z]^{\mathfrak{G}}$.
Due to the significance of this group $\mathfrak{G}$, we are led to reveal possible properties of $\mathfrak{G}$.
In \cite{MO} we defined the E-polynomials of $\mathfrak{G}$ and determined the ring generated by them. 
E-polynomials are finite analogue of Eisenstein series.
This paper stands on the same line.

On the other hand, we have 
a famous book \cite{We} by H.~Weyl which intensively studies
the invariant theory as well as the representation theory.
The centralizer algebra plays an important role in the arguments of this book.
We apply this philosophy to $\mathfrak{G}$
in order to get more properties.
The correspondence between the invariant forms and the representation is as follows:
$\mathfrak{G}$ acts on 
covariant vectors $y^{(1)},\ldots,y^{(k)}$ cogrediently and on 
contravariant vectors $\xi^{(1)},\ldots,\xi^{(k)}$ contragrediently.
Then the matrices
$B=(b(i_1\cdots i_k; j_1\cdots j_k))$ of the coefficients of the invariant forms 
\[
	\sum_{i;j}
	b(i_1\cdots i_k; j_1\cdots j_k)
	\xi_{i_1}^{(1)}\cdots \xi_{i_k}^{(k)}
	y_{j_1}^{(1)}\cdots y_{j_k}^{(k)}
\]
of $\mathfrak{G}$
form the centralizer algebra
of the $k$-th tensor representation of $\mathfrak{G}$.
In this paper we mainly discuss the centralizer algebras which is nothing else but the study of the invariant forms.




We organize this paper as follows.
First in Section~2
we give a presentation of ${\mathfrak{G}}$ by generators and relations.
To determine the centralizer algebras,
it is enough to consider the projective group
$P{\mathfrak{G}} = {\mathfrak{G}}/Z({\mathfrak{G}})$,
the coset group by the center. 
We give the presentation of $P{\mathfrak{G}}$ in Section~3.
In Section~4, we 
present
a complete set of 
all the equivalence classes of the irreducible representations of $P{\mathfrak{G}}$.
By Schur-Weyl duality,
the multiplicities of the irreducible representations
in the tensor space, determine the structure of the centralizer algebra.
So we decompose the tensor representation of $P{\mathfrak{G}}$ into
irreducible ones in Section~5.
Finally in Section~6,
we give the structure of the centralizer algebras of the tensor representations of $P{\mathfrak{G}}$.
\section{Characterization}

Let $\mathfrak{G}$ be the group introduced in Section 1.
We note that ${\cal T}^2 = \mbox{diag}(i, i, i)$
(and hence ${\cal T}^2$, 
${\cal T}^4$ and ${\cal T}^6$)
is in the center of $\mathfrak{G}$.
We also note that both ${\cal D}$ and ${\cal T}$ have order 8 and that
they satisfy the following relations:
\[
	{\cal T}{\cal D}{\cal T} = {\cal D}^7{\cal T}^3{\cal D}^7,\ 
	{\cal T}{\cal D}^5{\cal T} = {\cal D}^3{\cal T}^7{\cal D}^3.
\]
We show that actually these relations determine the group $\mathfrak{G}$, that is,

\begin{thm}\label{thm:defrelG}
Let 
$G$
be the group generated by
the symbols $D$ and $T$ which obey the following relations:
\begin{gather}
	D^8 = 1,\tag{R1}\\
	T^8 = 1,\tag{R2}\\
	T^2D = DT^2,\tag{R3}\\
	TDT = D^7T^3D^7,\tag{R4}\\
	TD^5T = D^3T^7D^3.\tag{R5}
\end{gather}
Then 
each element of $G$ can be written in exactly one of the 
following forms:
\begin{gather}
1\tag{W1}\\
D^{n_1},\tag{W2}\\
T^{n_2},\tag{W3}\\
D^{n_3}T^{n_4},\tag{W4}\\
T^{p_1}D^{n_5},\tag{W5}\\
D^{n_6}T^{p_2}D^{n_7},\tag{W6}\\
TD^{e_1}T^{p_3},\tag{W7}\\
D^{n_8}TD^{e_2}T^{p_2}\tag{W8}.
\end{gather}
Here $n_1, \ldots, n_8 \in \{1,2,\ldots, 7\}$,
$p_1, p_2, p_3 \in \{1, 3, 5, 7\}$
and $e_1, e_2 \in \{2, 4, 6\}$.
In particular the order of $G$ is $384$.
\end{thm}
Before proving the theorem above,
we shall show
\begin{lem}\label{lem:DTrel}
We have
\begin{align}
TD^3T &= D^5T^5D^5,\tag{R6}\\
TD^7T &= DTD,\tag{R7}\\
D^{2i}TD^{2j}T &= TD^{2j}TD^{2i},\quad 2i, 2j\in \{2,4,6\}.\tag{R8}
\end{align}
\end{lem}
\begin{proof}
In the following computations, 
the relations (R1), (R2) and (R3) are repeatedly used without mentioning.
From the relation (R5) we have
\[
TD^5T = D^3T^7D^3
\Leftrightarrow TD^5 = D^3T^7D^3T^7
\Leftrightarrow D^5TD^5 = T^7D^3T^7 = TD^3T^5.
\]
Multiplying $T^4$ from the right, we have
$D^5T^5D^5 = TD^3T$.

From the relation (R4) we have
\[
TDT = D^7T^3D^7 \Leftrightarrow TD = D^7T^3D^7T^7 = D^7TD^7T
\Leftrightarrow DTD = TD^7T.
\]
Finally we show (R8).
For the case $2i=2$, we have
\begin{align*}
D^2TD^2T &= D\cdot DTD\cdot DT = D\cdot TD^7T \cdot DT = DTD\cdot D^6\cdot TDT\\	&= TD^7T\cdot D^6\cdot D^7T^3D^7\\
	&= TD^7T\cdot D^{13}\cdot T^3D^7\\
	&= TD^7\cdot TD^5T\cdot T^2D^7\\
	&= TD^7\cdot D^3T^7D^3\cdot T^2D^7\\
	&= TD^{10}T^7T^2D^3D^7\\
	&= TD^2TD^2,\\
D^2TD^4T &= D\cdot DTD\cdot D^3T
= D\cdot TD^7T \cdot D^3T = DTD\cdot D^6\cdot TD^3T\\
	&= TD^7T\cdot D^6\cdot D^5T^5D^5\\
	&= TD^7T\cdot D^{11}\cdot T^5D^5\\
	&= TD^7\cdot TD^3T\cdot T^4D^5\\
	&= TD^7\cdot D^5T^5D^5\cdot T^4D^5\\
	&= TD^{12}T^5T^4D^5D^5\\
	&= TD^4TD^2
\end{align*}
and
\begin{align*}
D^2TD^6T &= D\cdot DTD\cdot D^5T
= D\cdot TD^7T \cdot D^5T = DTD\cdot D^6\cdot TD^5T\\
	&= TD^7T\cdot D^6\cdot D^3T^7D^3\\
	&= TD^7T\cdot D^9\cdot T^7D^3\\
	&= TD^7\cdot TDT\cdot T^6D^3\\
	&= TD^7\cdot D^7T^3D^7\cdot T^6D^3\\
	&= TD^{14}T^3T^6D^7D^3\\
	&= TD^6TD^2.
\end{align*}
For $2i=4$, we utilize the case $2i=2$ as follows.
\begin{align*}
D^4TD^2T &= D^2\cdot D^2TD^2T = D^2\cdot TD^2TD^2 = D^2TD^2T\cdot D^2
=TD^2TD^2\cdot D^2\\
	&=TD^2TD^4.
\end{align*}
Similarly for $2i=6$, we utilize the case $2i=4$
and we can obtain the remaining relations.
\end{proof}
\begin{rem}
The initial relations (R4), (R5) and the relations (R6), (R7) above
show that
\begin{equation*}
	TD^{\mbox{odd}}T = D^{\mbox{odd}}T^{\mbox{odd}}D^{\mbox{odd}}.
\end{equation*}
Also we have
\begin{equation*}
	TD^{\mbox{odd}}
= D^{\mbox{odd}}T^{\mbox{odd}}D^{\mbox{odd}}T^{\mbox{odd}}.
\end{equation*}
\end{rem}
\begin{proof}[Proof of Theorem~\ref{thm:defrelG}]
Suppose that $w\in G$ is written in one of the forms (W1)-(W8).
We show that all the left and the right multiplications
by the letters $D$ and $T$
are again written in one of the forms (W1)-(W8).

{\bf Case 1.} The word in the form (W1) is the identity.
It is obvious that the left and the right multiplications by $D$ and $T$
are written in the form (W2) and (W3) respectively.

{\bf Case 2.} Let $w = D^{n_1}$.
Then $Dw = wD = D^{n_1+1}$ is in the form (W2) or (W1).
The left $T$-action $Tw = TD^{n_1}$ is in the form (W5)
and the right $T$-action $wT = D^{n_1}T$ is in the form (W4).

{\bf Case 3.} Let $w = T^{n_2}$.
Then the left $D$-action $Dw = DT^{n_2}$ is in the form (W4).
The right $D$-action $wD = T^{n_2}D$ is in the form (W5)
if $n_2$ is odd,
and $T^{n_2}D = DT^{n_2}$ is in the form (W4) if $n_2$ is even.
The left and the right $T$-actions $Tw = wT = T^{n_2+1}$ is
in the form (W3) or (W1).

{\bf Case 4.} 
Let $w = D^{n_3}T^{n_4}$.
Then the left $D$-action $Dw = D^{n_3+1}T^{n_4}$ is in the form (W4).
The right $D$-action $wD = D^{n_3}T^{n_4}D$ is in the form (W6)
if $n_4$ is odd.
If $n_4$ is even, then $D^{n_3}T^{n_4}D = D^{n_3}DT^{n_4}$.
This is in the form (W4).
Next consider the left $T$-action.
If $n_4$ is even, then $Tw = TD^{n_3}T^{n_4} = TT^{n_4}D^{n_3}$.
This is in the form (W5). So we assume that $n_4$ is odd.
If $n_3$ is even, then $TD^{n_3}T^{n_4}$ is in the form (W7).
If $n_3$ is odd, then
\[
	TD^{n_3}T^{n_4}
	= TD^{n_3}T\cdot T^{\mbox{even}}
	= D^{\mbox{odd}}T^{\mbox{odd}}D^{\mbox{odd}}T^{\mbox{even}}
	= D^{\mbox{odd}}T^{\mbox{odd}}T^{\mbox{even}}D^{\mbox{odd}}.
\]
This is in the form (W6).
If $n_3$ is even and $n_4$ is odd,
then $TD^{n_3}T^{n_4}$ is in the form (W7).
The right $T$-action $wT = D^{n_3}T^{n_4}T$ is in the form (W4).

{\bf Case 5.} 
Let $w = T^{p_1}D^{n_5}$.
The left $D$-action $Dw = DT^{p_1}D^{n_5}$ is in the form (W6)
and the right $D$-action $wD = T^{p_1}D^{n_5}D$ is in the form (W5).
The left $T$-action
$Tw = TT^{p_1}D^{n_5} = T^{p_1+1}D^{n_5} = D^{n_5}T^{p_1+1}$.
This is in the form (W4).
Consider the right $T$-action $wT = T^{p_1}D^{n_5}T$.
If $n_5$ is even, this is in the form (W7).
If $n_5$ is odd, then
\[
T^{p_1}D^{n_5}T
= T^{p_1-1}\cdot TD^{n_5}T
= T^{p_1-1}\cdot D^{\mbox{odd}}T^{\mbox{odd}}D^{\mbox{odd}}
= D^{\mbox{odd}}T^{p_1-1}T^{\mbox{odd}}D^{\mbox{odd}}.
\]
This is in the form (W3) or (W6).

{\bf Case 6.} 
Let $w = D^{n_6}T^{p_2}D^{n_7}$.
Obviously, both the left and the right $D$-actions are in the form (W6).
Consider the the left $T$-action $Tw = TD^{n_6}T^{p_2}D^{n_7}$.
If both $D^{n_6}$ and $D^{n_7}$ are even,
then by (R8),
\begin{align*}
	TD^{n_6}T^{p_2}D^{n_7}
	&= TD^{n_6}TT^{p_2-1}D^{n_7}
	= TD^{n_6}TD^{n_7}\cdot T^{p_2-1}
	= D^{n_6}TD^{n_7}T\cdot T^{p_2-1}\\
	&= D^{n_6}TD^{n_7}T^{p_2}.
\end{align*}
This is in the form (W8).
If $n_6$ is odd, then
\begin{align*}
	TD^{n_6}T^{p_2}D^{n_7}
	&= TD^{n_6}T\cdot T^{p_2-1}D^{n_7}
	= D^{\mbox{odd}}T^{\mbox{odd}}D^{\mbox{odd}}T\cdot T^{p_2-1}D^{n_7}\\
	&= D^{\mbox{odd}}T^{\mbox{odd}}T^{p_2-1}D^{\mbox{odd}}T\cdot D^{n_7}.
\end{align*}
This is in the form (W6).
If $n_6$ is even and $n_7$ is odd, then we have
\begin{align}\label{eq:W6}
	TD^{n_6}T^{p_2}D^{n_7}
&= TD^{n_6}TT^{p_2-1}D^{n_7}\\\nonumber
&= TD^{n_6}\cdot TD^{n_7}\cdot T^{p_2-1}\\\nonumber
&= TD^{n_6}\cdot
D^{\mbox{odd}}T^{\mbox{odd}}D^{\mbox{odd}}T^{\mbox{odd}}\cdot T^{p_2-1}\\\nonumber
&= TD^{n_6}D^{\mbox{odd}}T\cdot
T^{\mbox{odd}-1}D^{\mbox{odd}}T^{\mbox{odd}}T^{p_2-1}\\\nonumber
&= D^{\mbox{odd}}T^{\mbox{odd}}D^{\mbox{odd}}\cdot
D^{\mbox{odd}}T^{\mbox{odd}-1}T^{\mbox{odd}}T^{p_2-1}\\\nonumber
&= D^{\mbox{odd}}TT^{\mbox{odd}-1}D^{\mbox{odd}}
D^{\mbox{odd}}T^{\mbox{odd}-1}T^{\mbox{odd}}T^{p_2-1}\\\nonumber
&= D^{\mbox{odd}}TD^{\mbox{odd}}D^{\mbox{odd}}
T^{\mbox{odd}-1}T^{\mbox{odd}-1}T^{\mbox{odd}}T^{p_2-1}.\nonumber
\end{align}
This is in the form (W8).
Consider the right $T$-action $wT = D^{n_6}T^{p_2}D^{n_7}T$.
If $n_7$ is even, then we have
$D^{n_6}T^{p_2}D^{n_7}T = D^{n_6}T\cdot T^{p_2-1}D^{n_7}T
= D^{n_6}TD^{n_7}T^{p_2-1}T$.
This is in the form $(W8)$.
If $n_7$ is odd, then we have
$D^{n_6}T^{p_2}D^{n_7}T = D^{n_6}T^{p_2-1}\cdot TD^{n_7}T
= D^{n_6}T^{p_2-1}D^{\mbox{odd}}T^{\mbox{odd}}D^{\mbox{odd}}
= D^{n_6}D^{\mbox{odd}}T^{p_2-1}T^{\mbox{odd}}D^{\mbox{odd}}
$.
This is in the form $(W6)$.

{\bf Case 7.} 
Let $w = TD^{e_1}T^{p_3}$.
The left $D$-action $Dw = DTD^{e_1}T^{p_3}$ is in the form (W8).
Consider the right $D$-action $wD = TD^{e_1}T^{p_3}D$.
We have already considered this case in the equation~\eqref{eq:W6}.
The left $T$-action $Tw = TTD^{e_1}T^{p_3} = D^{e_1}T^{p_3+2}$
is in the form (W4).
The right $T$-action
$wT = TD^{e_1}T^{p_3}T = TD^{e_1}T^{p_3+1} = TT^{p_3+1}D^{e_1}$
is in the form (W5).

{\bf Case 8.} 
Finally,
consider the case $w = D^{n_8}TD^{e_2}T^{p_4}$.
The left $D$-action $Dw = DD^{n_8}TD^{e_2}T^{p_4}$
is obviously in the form (W8).
Consider the right $D$-action $wD = D^{n_8}TD^{e_2}T^{p_4}D$.
As we have seen in the equation~\eqref{eq:W6},
$TD^{e_2}T^{p_4}D$ is in the form (W8).
So is $D^{n_8}\cdot TD^{e_2}T^{p_4}D$.
Consider the left $T$-action $Tw = TD^{n_8}TD^{e_2}T^{p_4}$.
If $n_8$ is even,
then we have
\[
	TD^{n_8}TD^{e_2}\cdot T^{p_4} = D^{e_2}TD^{n_8}T\cdot T^{p_4}
= D^{e_2}TD^{n_8}T^{p_4+1}
= D^{e_2}TT^{p_4+1}D^{n_8}.
\]
This is in the form (W6).
If $n_8$ is odd, then we have
\begin{align*}
	TD^{n_8}T\cdot D^{e_2}T^{p_4}
&= D^{\mbox{odd}}T^{\mbox{odd}}D^{\mbox{odd}}\cdot D^{e_2}T^{p_4}
= D^{\mbox{odd}}TT^{\mbox{odd}-1}D^{\mbox{odd}}D^{e_2}T^{p_4}\\
&= D^{\mbox{odd}}TD^{\mbox{odd}}D^{e_2}T^{\mbox{odd}-1}T^{p_4}.
\end{align*}
This is in the form (W8).
Finally consider the right $T$-action
$wT = D^{n_8}TD^{e_2}T^{p_4}T
= D^{n_8}TD^{e_2}T^{p_4+1} =D^{n_8}TT^{p_4+1}D^{e_2}$.
This is in the form (W6).

Hence we found that under the relations (R1)-(R5),
all possible words in the alphabet $\{D, T\}$
are written in the forms (W1)-(W8).

Consider the map $\iota$ from $G$ to $\mathfrak{G}$
defined by $D\mapsto{\cal D}$ and $T\mapsto{\cal T}$.
This is a well-defined group homomorphism,
since $\mathfrak{G}$ respects the relations (R1)-(R5).
This $\iota$ is surjective from the definition.
It easy to check that all images of the words 
(W1)-(W8)
are distinct. Hence $\iota$ is a group isomorphism.
\end{proof}

By the derived isomorphism in the course of the proof above, 
we identify $G$ with $\mathfrak{G}$.

\section{Projective Group}
As in the previous section,
let $G$ be the group of order $384$ generated by $D$ and $T$.
We are interested in the specified irreducible representation (assigned $\rho_7$ below) of $G$ of degree $3$ 
given explicitly in Introduction.
There some relationships to other fields such as the invariant theory and theory of modular forms are described.
We denote by $V$ a natural module of $\rho_7$.
Our objective is to analyze the structure of
the centralizer algebra $\mbox{End}_G(V^{\otimes k})$ of $G$
in the tensor space.
This result should be regarded as adding another aspect of this important group.
Since the center $Z$ of $G$ does not affect the structure of the centralizer,
in the following, we have only to consider the projective group
$PG = G/Z$ of $G$.

It is easy to see that $Z = \{1, T^2, T^4, T^6\}$.
By the definition above,
$PG$ is generated by $\bD = DZ$ and $\bT = TZ$.
By a similar argument to that of the previous section,
we find the defining relations
and all words of $PG$.
\begin{thm}
$PG = G/Z$ has the following presentation:
\\
generators:
\[
\bD, \bT
\]
and relations
\begin{equation}
	\bD^8 = 1,\ \bT^2 = 1,\ \bT\bD\bT = \bD^7\bT\bD^7,\ 
	\bT\bD^5\bT = \bD^3\bT\bD^3.\tag{R0'}
\end{equation}
Further, each element in $PG$ can be written
in exactly one of the following forms.
\begin{gather}
\bar{1},
\bD^{n_1},
\bT,
\bD^{n_1}\bT,
\bT\bD^{n_2},
\bD^{n_3}\bT\bD^{n_4},
\bT\bD^{e_1}\bT,
\bD^{n_5}\bT\bD^{e_2}\bT.
\end{gather}
Here $n_1, \ldots, n_5 \in \{1,2,\ldots, 7\}$,
and $e_1, e_2 \in \{2, 4, 6\}$.
\end{thm}

Further, we can find that $PG$ is divided into 10 conjugacy classes $\mathfrak{C}_1,\mathfrak{C}_2,\ldots,\mathfrak{C}_{10}$,
each of which is represented by 
\[
 \bar{1},\bD,\bD^2,\bD^3,\bD^4,\bD^6,\bT,\bD\bT,\bD^4\bT,\bD^2\bT\bD^4\bT.
\]

\section{The irreducible representations and the character table}
In this section, we will find all the irreducible representations of
$PG$ and its character table.
First we note that if we put $\bD = \eta D$ and $\bT = \eta^3 T$,
then $\bD$ and $\bT$ satisfy the relations (R0').
In other words, the following map $\rho_7$ affords a representation of $PG$:
\[
	\rho_7(\bD) =\eta D = \mbox{diag($\eta, i, -\eta$)},\quad
	\rho_7(\bT) = \eta^3 T = 
	\frac{-1}{2}
	\begin{pmatrix}
		1& 2&1\\
		1& 0&-1\\
		1&-2&1
	\end{pmatrix}.
\]

We find the complex conjugate
$\rho_6(\bD) = \bar{\rho_7}(\bD) = \mbox{diag}(\eta^7, -i, \eta^3)$
and $\rho_6(\bT) = \bar{\rho_7}(\bT) = \rho_7(\bT)$
also afford an irreducible representation of $PG$.
Moreover, the relations (R0')
preserve the parity of the number of generators,
we find that $\rho_2(\bD) = \rho_2(\bT) = -1$
afford a 1-dimensional representation other than
a trivial representation $\rho_1$.
Hence we have further two 3-dimensional irreducible
representations, $\rho_8 = \rho_2\otimes\rho_7$
and $\rho_9 = \rho_2\otimes\rho_6$.

Next we consider $\rho_7\otimes\rho_7$ and $\rho_7\otimes\rho_6$.
Let
$\langle \bolde_i\otimes\bolde'_j\ |\ i, j = 1,2,3\rangle$
be a basis of $V_7\otimes V_7$ (resp. $V_7\otimes V_6$).
Then we have the following representation matrices of $\bD$
and $\bT$ respectively:
\begin{align*}
	\rho_{7}\otimes\rho_{7}(\bD)
	&= \mbox{diag}(i, \eta^3, -i,
				   \eta^3,     -1, \eta^7,
				   -i, \eta^7, i)\\
	(\mbox{resp.} \rho_{7}\otimes\rho_{6}(\bD)
	&= \mbox{diag}(1, \eta^7, -1,
				   \eta,     1, \eta^5,
				   -1, \eta^3, 1)),\\
	\rho_{7}\otimes\rho_{7}(\bT) &= \rho_{7}\otimes\rho_{6}(\bT)\nonumber\\
	&= \frac{1}{4}
	\begin{pmatrix}
		1& 2& 1& 2& 4& 2& 1& 2& 1\\
		1& 0&-1& 2& 0&-2& 1& 0&-1\\
		1&-2& 1& 2&-4& 2& 1&-2& 1\\
		1& 2& 1& 0& 0& 0&-1&-2&-1\\
		1& 0&-1& 0& 0& 0&-1& 0& 1\\
		1&-2& 1& 0& 0& 0&-1& 2&-1\\
		1& 2& 1&-2&-4&-2& 1& 2& 1\\
		1& 0&-1&-2& 0& 2& 1& 0&-1\\
		1&-2& 1&-2& 4&-2& 1&-2& 1
	\end{pmatrix} .
\end{align*}
If we put
$\boldv_1 = \bolde_1\otimes \bolde'_1 + \bolde_3\otimes \bolde'_3$,
$\boldv_2 = \bolde_1\otimes \bolde'_3 + \bolde_3\otimes \bolde'_1$
and $\boldv_3 = \bolde_2\otimes \bolde'_2$,
then we find $V_4=\langle \boldv_1, \boldv_2, \boldv_3\rangle$
is $\rho_7^{\otimes 2}(\bT)$-invariant
as well as $\rho_7^{\otimes 2}(\bD)$-invariant.
Indeed, if we put $\rho_4 = \rho_7^{\otimes 2}|_{V_4}$,
then we have 
\[
\rho_4(\bT)(\boldv_1, \boldv_2, \boldv_3)
=(\boldv_1, \boldv_2, \boldv_3)
\frac{1}{2}
	\begin{pmatrix}
		1&1&2\\
		1&1&-2\\
		1&-1&0
	\end{pmatrix}
\]
and
\[
\rho_4(\bD)(\boldv_1, \boldv_2, \boldv_3)
= (\boldv_1, \boldv_2, \boldv_3)\mbox{diag}(i, -i, -1).
\]
Note that no eigenvector of $\rho_4(\bD)$ is 
$\rho_4(\bT)$-invariant.
This implies that $(\rho_4, V_4)$ defines an irreducible
representation of $PG$.
Hence $\rho_5 = \rho_4\otimes\rho_2$ also defines
an irreducible representation.

Next we put $V_3 = \langle \boldv_4, \boldv_5\rangle$, where
$\boldv_4 =
	2\bolde_1\otimes \bolde'_3 + 2\bolde_3\otimes \bolde'_1$ and 
\[
\boldv_5
	= \bolde_1\otimes\bolde'_1
	+ \bolde_3\otimes\bolde'_3
	- \bolde_1\otimes\bolde'_3
	- \bolde_3\otimes\bolde'_1
	- \bolde_2\otimes\bolde'_2.
\]
Then $\rho_3 = \rho_7\otimes\rho_6 |_{V_3}$
defines an irreducible representation of degree 2.
The representation matrices with respect to this basis
are
\[
\rho_3(\bD)(\boldv_4, \boldv_5)
=	(\boldv_4, \boldv_5)
\begin{pmatrix}
		-1&1\\
		0&1
	\end{pmatrix},\quad
\rho_3(\bT)(\boldv_4, \boldv_5)
=	(\boldv_4, \boldv_5)
\begin{pmatrix}
		1&0\\
		1&-1
	\end{pmatrix}.
\]
It is easy to check that this representation is irreducible.

So far, we have got 9 irreducible representations of $PG$.
The square sum of the dimensions of these irreducible representations is
\[
	1^2+1^2+2^2+3^2+3^2+3^2+3^2+3^2+3^2 = 60.
\]
Since $PG$ is of order 96 and has 10 conjugacy classes,
there must be one more irreducible representation $(\rho_{10}, V_{10})$
of degree 6.
Although we have not yet obtained the final representation,
if we use the orthogonality of the characters
we can find $\chi_{10}$, the character of $\rho_{10}$,
and obtain the character table of $PG$ as follows.
\begin{center}
\begin{tabular}{c|cccccccccc}
\hline
$PG$&	$\mathfrak{C}_1$&$\mathfrak{C}_2$&$\mathfrak{C}_3$&$\mathfrak{C}_4$&
		$\mathfrak{C}_5$&$\mathfrak{C}_6$&$\mathfrak{C}_7$&$\mathfrak{C}_8$&
	$\mathfrak{C}_9$&   $\mathfrak{C}_{10}$\\
     &       $\bar{1}$&     $\bD$&   $\bD^2$&   $\bD^3$&
	   $\bD^4$&   $\bD^6$&     $\bT$&    $\bD\bT$&
	  $\bD^4\bT$&  $\bD^2\bT\bD^4\bT$\\
order	& 1& 8& 4& 8&
	  2& 4& 2& 3&
	  4&4\\
size	& 1& 12& 3& 12&
	  3& 3& 12& 32&
	  12&6\\
\hline
$\chi_1$&
	 1&1&1&1&
	 1&1&1&1&
	 1&1\\
$\chi_2$&
	  1&$-1$&   1&$-1$&
	  1&   1&$-1$&   1&
	$-1$&   1\\
$\chi_3$&
	  2&0&2&0&
	  2&2&0&$-1$&
	  0&2\\
$\chi_4$&
	  3& $-1$&$-1$&$-1$&
	  3&$-1$&1&0&
	  1&$-1$\\
$\chi_5$&
	  3& $1$&$-1$&$1$&
	  3&$-1$&$-1$&0&
	  $-1$&$-1$\\
$\chi_6$&
	  3&$-i$&$a$&$i$&
	  $-1$&$b$&$-1$&0&
	  1&$1$\\
$\chi_7$&
	  3&$i$&$b$&$-i$&
	  $-1$&$a$&$-1$&0&
	  1&$1$\\
$\chi_8$&
	  3&$-i$&$b$&$i$&
	  $-1$&$a$&$1$&0&
	  $-1$&$1$\\
$\chi_9$&
	  3&$i$&$a$&$-i$&
	  $-1$&$b$&$1$&0&
	  $-1$&$1$\\
$\chi_{10}$& 6&    0&  2&   0&
	  $-2$&$2$& 0&   0&
	  0&  $-2$
\end{tabular}
\end{center}
Here $a = -1-2i$ and $b = -1+2i$.

Now we go back to the representation
$\rho_7\otimes\rho_6$.
If we put
\begin{align*}
V_{10} &=
\langle
	\boldw_3= \bolde_1\otimes\bolde'_1-\bolde_3\otimes\bolde'_3,
	\boldw_4= \bolde_1\otimes\bolde'_3-\bolde_3\otimes\bolde'_1,\\
	&\qquad
	\boldw_5= \bolde_1\otimes\bolde'_2,
	\boldw_6= \bolde_2\otimes\bolde'_1,
	\boldw_7= \bolde_2\otimes\bolde'_3,
	\boldw_8= \bolde_3\otimes\bolde'_2
\rangle,
\end{align*}
then $V_{10}$ is both $\bT$-invariant and $\bD$-invariant.
Hence $V_{10}$ defines a representation $\rho$
and we have
\[
	\rho(\bT)(\boldw_3, \ldots, \boldw_8)
= (\boldw_3, \ldots, \boldw_8)
\dfrac{1}{2}
\begin{pmatrix}
	0& 0& 1& 1& 1& 1\\
	0& 0&-1& 1& 1&-1\\
	1&-1& 0& 1&-1& 0\\
	1& 1& 1& 0& 0&-1\\
	1& 1&-1& 0& 0& 1\\
	1&-1& 0&-1&1& 0
\end{pmatrix}
\]
and
\[
	\rho(\bD)(\boldw_3, \ldots, \boldw_8)
	= \mbox{diag$(1,-1,\eta^7,\eta,\eta^5,\eta^3)$}.
\]
Since every character value of $\rho$ at each conjugacy class
coincides with that of $\chi_{10}$ of the character table of $PG$,
$\rho = \rho_{10}$ gives the irreducible representation
and we have thus obtained a complete set of representatives of the
irreducible representations.

\section{Decomposition of tensor representations}
In the previous section,
we have found the complete set of representatives of
all irreducible representations of $PG$.
In this section, we examine how the tensor powers of $\rho_{7}$
are decomposed into irreducible ones.

We follow the argument presented in
the paper \cite{KO}.
Let $\chi_1,\ldots, \chi_{10}$ be the
complete set of all irreducible characters of the group $PG$.
Now suppose that we have a character 
\[
\chi=m_1\chi_1+m_2\chi_2+\cdots + m_{10}\chi_{10},
\]
which
has a value $k_i$ 
at each conjugacy class $\mathfrak{C}_i,\ i=1,2,\ldots,10$.
Let $\boldX$ denote the matrix of the character table of $PG$,
then we have
\begin{equation}\label{eq:MulIrr}
	(m_1, \ldots, m_{10}) = (k_1, \ldots, k_{10})\boldX^{-1}.
\end{equation}

In order to examine the structure of the centralizer algebra
of $\rho_7^{\otimes k}$,
the tensor power of the natural representation $\rho_7$,
we decompose it into the irreducible ones.
For this, we need to decompose
$\rho_{7}\otimes\rho_i$ ($i = 1, 2, \ldots, 10$)
one by one.

By the argument and/or the character table in the previous section,
we already have the following:
\begin{align*}
	\chi_{7}\cdot\chi_1 &= \chi_{7},\\
	\chi_{7}\cdot\chi_2 &= \chi_{8}.
\end{align*}
Using the equation~\eqref{eq:MulIrr}, further we have
\begin{align*}
	\chi_{7}\cdot\chi_3 &= \chi_7+\chi_8,\\
	\chi_{7}\cdot\chi_4 &= \chi_6+\chi_{10},\\
	\chi_{7}\cdot\chi_5 &= \chi_9+\chi_{10},\\
	\chi_{7}\cdot\chi_6 &= \chi_1+\chi_3+\chi_{10},\\
	\chi_{7}\cdot\chi_7 &= \chi_4+\chi_6+\chi_9,\\
	\chi_{7}\cdot\chi_8 &= \chi_5+\chi_6+\chi_9, \\
	\chi_{7}\cdot\chi_9 &= \chi_2+\chi_3+\chi_{10},\\
	\chi_{7}\cdot\chi_{10} &= \chi_4+\chi_5+\chi_7+\chi_8+\chi_{10}.
\end{align*}

For example, we take up the $\chi_7\cdot \chi_{10}$ case.
It is easy to compute
\[
\chi_7\cdot\chi_{10}(\mathfrak{C}_{1}, \ldots, \mathfrak{C}_{10})
= (18, 0, -2+4i, 0, 2, -2-4i, 0, 0, 0, -2).
\]
Then we compute
\[
(18, 0, -2+4i, 0, 2, -2-4i, 0, 0, 0, -2)\boldX^{-1}
=(0, 0, 0, 1, 1, 0, 1, 1, 0, 1)
\]
and this means, due to the identity \eqref{eq:MulIrr},
\[
	\chi_{7}\cdot\chi_{10} = \chi_4+\chi_5+\chi_7+\chi_8+\chi_{10}.
\]

By the above calculation,
we obtain the Bratteli diagram of the decomposition of
$\rho_7^{\otimes k}$ into irreducible ones.
(For the Bratteli diagram,
see for example Goodman-de la Harpe-Jones~\cite{GHJ}, \S 2.3.)
Accordingly,
the square sum of the multiplicities on the $k$-th row
is the dimension of $\mbox{End}_{PG}(V_{7}^{\otimes k})$.
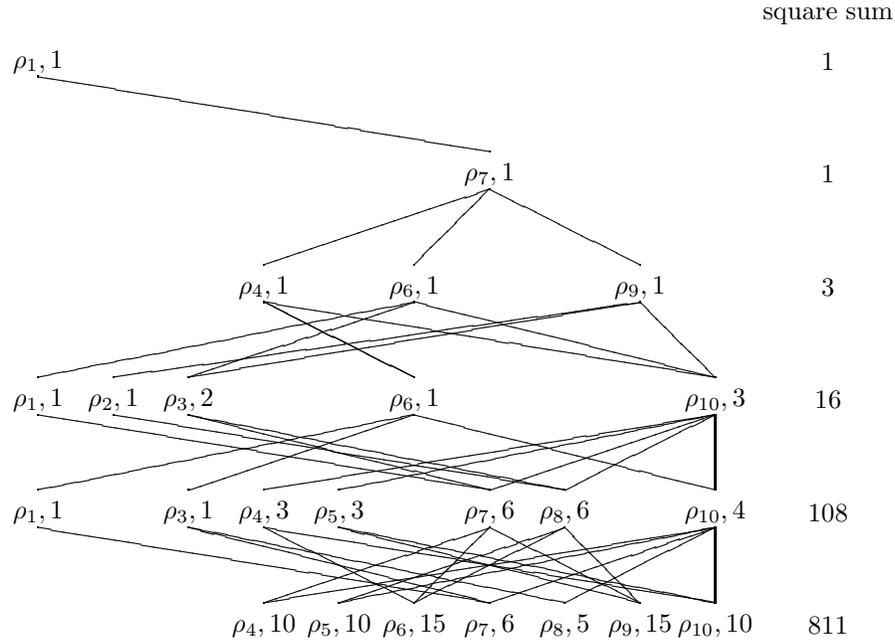
\[
\begin{xy}
(115,3)*{\mbox{square sum}}, 
(10,-5)="S0P1",
(10,-3)*{\rho_{1},1},
(20,-3)*{},
(30,-3)*{},
(40,-3)*{},
(50,-3)*{},
(60,-3)*{},
(70,-3)*{},
(80,-3)*{},
(90,-3)*{},
(100,-3)*{},
(115,-3)*{1}, 
(10,-18)*{},
(20,-18)*{},
(30,-18)*{},
(40,-18)*{},
(50,-18)*{},
(60,-18)*{},
(70,-18)*{\rho_{7},1}, (70,-15)="T1P7", (70,-20)="S1P7",
(80,-18)*{},
(90,-18)*{},
(100,-18)*{},
(115,-18)*{1}, 
(10,-33)*{},
(20,-33)*{},
(30,-33)*{},
(40,-33)*{\rho_{4},1}, (40,-30)="T2P4", (40,-35)="S2P4",
(50,-33)*{},
(60,-33)*{\rho_{6},1}, (60,-30)="T2P6", (60,-35)="S2P6",
(70,-33)*{},
(80,-33)*{},
(90,-33)*{\rho_{9},1}, (90,-30)="T2P9", (90,-35)="S2P9",
(100,-33)*{},
(115,-33)*{3}, 
(10,-48)*{\rho_{1},1}, (10,-45)="T3P1", (10,-50)="S3P1",
(20,-48)*{\rho_{2},1}, (20,-45)="T3P2", (20,-50)="S3P2",
(30,-48)*{\rho_{3},2}, (30,-45)="T3P3", (30,-50)="S3P3",
(40,-48)*{},
(50,-48)*{},
(60,-48)*{\rho_{6},1}, (60,-45)="T3P6", (60,-50)="S3P6",
(70,-48)*{},
(80,-48)*{},
(90,-48)*{},
(100,-48)*{\rho_{10},3}, (100,-45)="T3P10", (100,-50)="S3P10",
(115,-48)*{16}, 
(10,-63)*{\rho_{1},1}, (10,-60)="T4P1", (10,-65)="S4P1",
(20,-63)*{},
(30,-63)*{\rho_{3},1}, (30,-60)="T4P3", (30,-65)="S4P3",
(40,-63)*{\rho_{4},3}, (40,-60)="T4P4", (40,-65)="S4P4",
(50,-63)*{\rho_{5},3}, (50,-60)="T4P5", (50,-65)="S4P5",
(60,-63)*{},
(70,-63)*{\rho_{7},6}, (70,-60)="T4P7", (70,-65)="S4P7",
(80,-63)*{\rho_{8},6}, (80,-60)="T4P8", (80,-65)="S4P8",
(90,-63)*{},
(100,-63)*{\rho_{10},4}, (100,-60)="T4P10", (100,-65)="S4P10",
(115,-63)*{108}, 
(10,-78)*{},
(20,-78)*{},
(30,-78)*{},
(40,-78)*{\rho_{4},10}, (40,-75)="T5P4", (40,-80)="S5P4",
(50,-78)*{\rho_{5},10}, (50,-75)="T5P5", (50,-80)="S5P5",
(60,-78)*{\rho_{6},15}, (60,-75)="T5P6", (60,-80)="S5P6",
(70,-78)*{\rho_{7},6}, (70,-75)="T5P7", (70,-80)="S5P7",
(80,-78)*{\rho_{8},5}, (80,-75)="T5P8", (80,-80)="S5P8",
(90,-78)*{\rho_{9},15}, (90,-75)="T5P9", (90,-80)="S5P9",
(100,-78)*{\rho_{10},10}, (100,-75)="T5P10", (100,-80)="S5P10",
(115,-78)*{811}, 
{\ar @{{.}-{.}}"S0P1";"T1P7"}, 
{\ar @{{.}-{.}}"S1P7";"T2P4"}, 
{\ar @{{.}-{.}}"S1P7";"T2P6"},
{\ar @{{.}-{.}}"S1P7";"T2P9"},
{\ar @{{.}-{.}}"S2P4";"T3P6"}, 
{\ar @{{.}-{.}}"S2P4";"T3P10"},
{\ar @{{.}-{.}}"S2P6";"T3P1"},
{\ar @{{.}-{.}}"S2P6";"T3P3"},
{\ar @{{.}-{.}}"S2P6";"T3P10"},
{\ar @{{.}-{.}}"S2P9";"T3P2"},
{\ar @{{.}-{.}}"S2P9";"T3P3"},
{\ar @{{.}-{.}}"S2P9";"T3P10"},
{\ar @{{.}-{.}}"S2P4";"T3P6"},
{\ar @{{.}-{.}}"S3P1";"T4P7"}, 
{\ar @{{.}-{.}}"S3P2";"T4P8"},
{\ar @{{.}-{.}}"S3P3";"T4P7"},
{\ar @{{.}-{.}}"S3P3";"T4P8"},
{\ar @{{.}-{.}}"S3P6";"T4P1"},
{\ar @{{.}-{.}}"S3P6";"T4P3"},
{\ar @{{.}-{.}}"S3P6";"T4P10"},
{\ar @{{.}-{.}}"S3P10";"T4P4"},
{\ar @{{.}-{.}}"S3P10";"T4P5"},
{\ar @{{.}-{.}}"S3P10";"T4P7"},
{\ar @{{.}-{.}}"S3P10";"T4P8"},
{\ar @{{.}-{.}}"S3P10";"T4P10"},
{\ar @{{.}-{.}}"S4P1";"T5P7"}, 
{\ar @{{.}-{.}}"S4P3";"T5P7"},
{\ar @{{.}-{.}}"S4P3";"T5P8"},
{\ar @{{.}-{.}}"S4P4";"T5P6"},
{\ar @{{.}-{.}}"S4P4";"T5P10"},
{\ar @{{.}-{.}}"S4P5";"T5P9"},
{\ar @{{.}-{.}}"S4P5";"T5P10"},
{\ar @{{.}-{.}}"S4P7";"T5P4"},
{\ar @{{.}-{.}}"S4P7";"T5P6"},
{\ar @{{.}-{.}}"S4P7";"T5P9"},
{\ar @{{.}-{.}}"S4P8";"T5P5"},
{\ar @{{.}-{.}}"S4P8";"T5P6"},
{\ar @{{.}-{.}}"S4P8";"T5P9"},
{\ar @{{.}-{.}}"S4P10";"T5P4"},
{\ar @{{.}-{.}}"S4P10";"T5P5"},
{\ar @{{.}-{.}}"S4P10";"T5P7"},
{\ar @{{.}-{.}}"S4P10";"T5P8"},
{\ar @{{.}-{.}}"S4P10";"T5P10"},
\end{xy}
\]

\section{Centralizer algebra}
In the previous section,
we have seen that
the dimensions of ${\cal A}_k = \mbox{End}_{PG}(V_{7}^{\otimes k})$
($k=0,1,2,\ldots,5$)
are 1, 1, 3, 16, 108, 811.
The number of paths from the top vertex to
the bottom vertices on the Hasse diagram
will be calculated using the adjacency matrix of the diagram.
To be explicit, the formulae
\[
\chi_7(\mathfrak{C}_j)\cdot \chi_i(\mathfrak{C}_j)
=\sum_{k=1}^{10}m_{ik}\chi_k(\mathfrak{C}_j),
\ \ \  i,j=1,2,\ldots,10
\]
obtained in the previous section lead to 
\[
\boldX\mbox{diag$(\chi_7(\mathfrak{C}_1), \ldots, \chi_7(\mathfrak{C}_{10}))$}
 = A\boldX
\]
or
\begin{equation}\label{eq:CharAdjacency}
\boldX\mbox{diag$(\chi_7(\mathfrak{C}_1), \ldots, \chi_7(\mathfrak{C}_{10}))$}
\boldX^{-1} = A
\end{equation}
where
\[
A =
\begin{pmatrix}
0&0&0&0&0&0&1&0&0&0\\
0&0&0&0&0&0&0&1&0&0\\
0&0&0&0&0&0&1&1&0&0\\
0&0&0&0&0&1&0&0&0&1\\
0&0&0&0&0&0&0&0&1&1\\
1&0&1&0&0&0&0&0&0&1\\
0&0&0&1&0&1&0&0&1&0\\
0&0&0&0&1&1&0&0&1&0\\
0&1&1&0&0&0&0&0&0&1\\
0&0&0&1&1&0&1&1&0&1
\end{pmatrix}\, .
\]

Let $d_1(k),\ldots, d_{10}(k)$ be
the number of paths at the bottom (the $k$-th floor) vertices
of the Hasse diagram.
Here $k$ starts from $0$ and the floor on which the first $\rho_1$ lies is the $0$-th floor.
Then the $d_{\ell}(k)$s are equal to the multiplicities of the irreducible
components of $PG$ in End$(V_7^{\otimes k})$:
\[
\chi_7^k=d_1(k)\chi_1+d_2(k)\chi_2+\cdots + d_{10}(k)\chi_{10}.
\]
It is easy to see
\[
(d_1(k),\ldots, d_{10}(k)) =
	(1,0,0,0,0,0,0,0,0,0)A^k \quad (k\geq 0).
\]
By the relation \eqref{eq:CharAdjacency}, however,
they are calculated by the character table $\boldX$ of $PG$ as follows.
\[
(d_1(k),\ldots, d_{10}(k)) =
(1, 0,\ldots, 0)\boldX
\mbox{diag$(\chi_7(\mathfrak{C}_1)^k, \ldots, \chi_7(\mathfrak{C}_{10})^k)$}
\boldX^{-1}.
\]
Hence we have for $k\geq 1$
\begin{align*}
d_1(k) &= \frac{3^k}{96}+\frac{a^k}{32}+\frac{b^k}{32}
+\frac{5\cdot(-1)^k}{32}+\frac{3}{16}+\frac{i^k}{8}+\frac{(-i)^k}{8},\\
d_2(k) &= {\frac {3^k}{96}}+\frac{a^k}{32}+\frac{b^k}{32}
-\frac{3\cdot(-1)^k}{32}-\frac{1}{16}-\frac{i^k}{8}-\frac{(-i)^k}{8},\\
d_{3}(k) &= \frac{3^k}{48}+\frac{a^k}{16}
+\frac{b^k}{16}+\frac{(-1)^k}{16}+\frac{1}{8},\\
d_{4}(k) &= \frac{3^k}{32}-\frac{a^k}{32}-\frac{b^k}{32}
+\frac{7\cdot(-1)^k}{32}+\frac{1}{16}-\frac{i^k}{8}-\frac{(-i)^k}{8},\\
d_{5}(k) &= \frac{3^k}{32}-\frac{a^k}{32}-\frac{b^k}{32}
-\frac{(-1)^k}{32}-\frac{3}{16}+\frac{i^k}{8}+\frac{(-i)^k}{8},\\
d_{6}(k) &= \frac{3^k}{32}+\frac{a\cdot a^k}{32}
+\frac{b\cdot b^k}{32}-\frac{5\cdot(-1)^k}{32}
+\frac{3}{16}+\frac{i\cdot{i}^k}{8}-\frac{i\cdot(-i)^k}{8},\\
d_{7}(k) &= \frac{3^k}{32}+\frac{b\cdot a^k}{32}
+\frac{a\cdot b^k}{32}-\frac {5\cdot(-1)^k}{32}
+\frac{3}{16}-\frac{i\cdot{i}^k}{8}+\frac{i\cdot(-i)^k}{8},\\
d_{8}(k) &= \frac{3^k}{32}+\frac{b\cdot a^k}{32}
+\frac{a\cdot b^k}{32}+\frac{3\cdot(-1)^k}{32}
-\frac{1}{16}+\frac{i\cdot{i}^k}{8}-\frac{i\cdot(-i)^k}{8},\\
d_{9}(k) &= \frac{3^k}{32}+\frac{a\cdot a^k}{32}
+\frac{b\cdot b^k}{32}+\frac{3\cdot(-1)^k}{32}
-\frac{1}{16}-\frac{i\cdot{i}^k}{8}+\frac{i\cdot(-i)^k}{8},\\
d_{10}(k) &= \frac{3^k}{16}+\frac{a^k}{16}+\frac{b^k}{16}
-\frac{(-1)^k}{16}-\frac{1}{8}.
\end{align*}
Here $a = -1-2i$ and $b = -1+2i$.
We shall summarize our results in the following way.

\begin{thm}
Let ${\cal A}_k = \End_{PG}(V_7^{\otimes k})$
be a centralizer algebra of $PG$ in $V_7^{\otimes k}$,
where $PG$ acts on $V_7^{\otimes k}$ diagonally.
Then ${\cal A}_k$ has the following multi-matrix structure:
\[
	{\cal A}_{k}
	\cong
\begin{cases} \Comp & (k=0),\\
	\bigoplus_{\ell=1}^{10} 
	M_{d_{\ell}(k)}(\Comp) &(k\geq 1) \end{cases}
\]
in which the $d_{\ell}(k)$s are explicitly determined above
and $M_d(\Comp)$ denotes the matrix algebra over $\Comp$ of size $d$.
\end{thm}
Calculating the square sum of the dimensions of
the simple components of ${\cal A}_k$
we finally derive the following
\begin{cor}
We have
\[
\dim {\cal A}_k = \begin{cases} 1 & (k=0),\\
\dfrac{57+6\cdot 5^k+9^k}{96} & (k\geq 1).
\end{cases}
\]
\end{cor}

\bigskip

We conclude this paper with a small table of the values $\dim \mathcal{A}_k$.

\medskip

\begin{tabular}{|c|c|c|c|c|c|c|c|c|c|c|}
\hline
$k $& 0 & 1 & 2 & 3 & 4& 5 & 6 & 7 & 8 &9\\
\hline
$\dim \mathcal{A}_k$ & 1 & 1 & 3 & 16 & 108 & 811 & 6513 & 54706& 472818& 4157701\\
\hline
\end{tabular}

\bigskip

{\bf Acknowledgment}. 
This work was supported by JSPS KAKENHI Grant Number 25400014.

{\sc Department of Mathematical Sciences, University of the Ryukyus,
Okinawa, 903-0213,
Japan}

{\it E-mail address}: kosuda@math.u-ryukyu.ac.jp

\bigskip

{\sc Graduate School of Natural Science and Technology,
Kanazawa University,
Ishikawa, 920-1192
Japan}

{\it E-mail address}: oura@se.kanazawa-u.ac.jp

\begin{thebibliography}{99}


\bibitem{BTypeII}
Bannai,~E., Dougherty,~S., Harada,~M., Oura,~M.,
Type II codes, even unimodular lattices, and invariant rings,
IEEE Trans. Inform. Theory {\bf 45} (1999), no. 4, 1194-1205. 





\bibitem{BZ4}
Bonnecaze,~A., Sol\'{e},~P., Bachoc,~C., Mourrain,~B.,
Type II codes over ${\Z}_4$,
IEEE Trans. Inform. Theory {\bf 43} (1997), no. 3, 969-976. 







\bibitem{Cameron}
Cameron,~P.~J.,
Cycle index, weight enumerator, and Tutte polynomial,
Electron. J. Combin. {\bf 9} (2002), no. 1, Note 2, 10 pp. 




\bibitem{Gl}%
Gleason,~A.~M.,
Weight polynomials of self-dual codes and the MacWilliams identities,
Actes du Congr\`{e}s International des Math\'{e}maticiens (Nice, 1970),
Tome 3, pp. 211-215. 
Gauthier-Villars, Paris, 1971. 

\bibitem{GHJ}%
Goodman,~F.~M., de la Harpe,~P. and Jones,~V.~F.~R., 
{\em Coxeter Graphs and Towers of Algebras.}
Springer-Verlag, New York, 1989.


\bibitem{HZ4}
Hammons,~A.~R.,~Jr., Kumar,~P.~V., Calderbank,~A.~R., Sloane,~N.~J.~A., Sol\'{e},~P.,
The ${\Z}_4$-linearity of Kerdock, Preparata, Goethals, and related codes,
IEEE Trans. Inform. Theory {\bf 40} (1994), no. 2, 301-319. 

\bibitem{KO}%
Kosuda,~M. and Oura,~M.,
On the centralizer algebras of
the primitive unitary reflection group of order $96$,
Tokyo J.~Math {\bf 39} (2016), no. 2, 469-482.



\bibitem{MO}
Motomura,~T. and Oura,~M.,
E-polynomials associated to ${\Z}_4$-codes,
accepted for publication in Hokkaido Math.~J.

\bibitem{Nbook}
Nebe,~G., Rains,~E.~M., Sloane,~N.~J.~A.,
Self-dual codes and invariant theory,
Algorithms and Computation in Mathematics, 17. Springer-Verlag, Berlin, 2006. 


\bibitem{Ru2}%
Runge,~B.,
Codes and Siegel modular forms,
Discrete Math. {\bf 148} (1996), 175-204. 

\bibitem{Ru3}
Runge,~B.,
Theta functions and Siegel-Jacobi forms,
Acta Math. {\bf 175} (1995), no. 2, 165-196. 




\bibitem{We}%
Weyl,~H.,
{\em The Classical Groups. Their Invariants and Representations.}
Princeton University Press, Princeton,~N.~J., 1939. 
\end{thebibliography}
\end{document}